\documentclass[10pt]{amsart}
\usepackage{amssymb, amsthm} 
\usepackage[varg]{pxfonts}   
\usepackage{graphicx, psfrag}

\newtheorem{thm}{Theorem}
\newtheorem{lemma}{Lemma}
\newtheorem{corol}[thm]{Corollary}

\theoremstyle{definition}

\newtheorem*{ack}{Acknowledgement}

\theoremstyle{remark}
\newtheorem{rem}{Remark}

\DeclareMathOperator{\tr}{tr}
\def\R{\mathbb{R}}
\def\C{\mathbb{C}}
\def\Z{\mathbb{Z}}
\def\N{\mathbb{N}}

\def\T{\mathbb{T}}
\def\P{\mathbb{P}}
\def\D{\mathbb{D}}
\def\G{\mathbb{G}}
\def\H{\mathbb{H}}

\def\cV{\mathcal{V}}
\def\cW{\mathcal{W}}
\def\SL{\mathrm{SL}}
\def\GL{\mathrm{GL}}

\def\SO{{\mathrm{SO}}}
\def\PSL{{\mathrm{PSL}}}

\def\id{{\operatorname{id}}}
\def\Id{{\operatorname{Id}}}  
\DeclareMathOperator{\interior}{int}
\renewcommand{\epsilon}{\varepsilon}
\renewcommand{\setminus}{\smallsetminus}
\renewcommand{\emptyset}{\varnothing}

\def\DC{{\mathrm{DC}}}

\def\Ruth{\mathit{Ruth}} 

\newcommand{\comm}[1]{}

\begin{document}

\title[Cantor Spectrum for Generalized Skew-shifts]
{Cantor Spectrum for Schr\"odinger Operators with Potentials arising
from Generalized Skew-shifts}

\begin{abstract}
We consider continuous $\SL(2,\R)$-cocycles over a strictly
ergodic homeomorphism which fibers over an almost periodic
dynamical system (generalized skew-shifts).  We prove that any
cocycle which is not uniformly hyperbolic can be approximated by
one which is conjugate to an $\SO(2,\R)$-cocycle. Using this, we
show that if a cocycle's homotopy class does not display a certain
obstruction to uniform hyperbolicity, then it can be
$C^0$-perturbed to become uniformly hyperbolic. For cocycles
arising from Schr\"odinger operators, the obstruction vanishes and
we conclude that uniform hyperbolicity is dense, which implies
that for a generic continuous potential, the spectrum of the
corresponding Schr\"odinger operator is a Cantor set.
\end{abstract}

\author[Avila]{Artur Avila}
\address{CNRS UMR 7599,
Laboratoire de Probabilit\'es et Mod\`eles al\'eatoires.
Universit\'e Pierre et Marie Curie--Bo\^\i te courrier 188.
75252--Paris Cedex 05, France}
\curraddr{IMPA, Rio de Janeiro, Brazil}
\urladdr{www.proba.jussieu.fr/pageperso/artur/}
\email{artur@math.sunysb.edu}

\author[Bochi]{Jairo Bochi}
\address{Instituto de Matem\'atica, UFRGS, Porto Alegre, Brazil}
\curraddr{Departamento de Matem\'atica, PUC-Rio, Rio de Janeiro, Brazil}
\urladdr{www.mat.ufrgs.br/~jairo}
\email{jairo@mat.ufrgs.br}

\author[Damanik]{David Damanik}
\address{Department of Mathematics, Rice University, Houston, TX 77005, USA}
\urladdr{www.ruf.rice.edu/$\sim$dtd3} \email{damanik@rice.edu}

\thanks{This research was partially conducted during the period A.~A.\ served as a Clay
Research Fellow.
J.~B.\ was partially supported by a grant from CNPq--Brazil.
D.~D.\ was supported in part by NSF grant DMS--0653720.
We benefited from CNPq and Procad/CAPES support for traveling.}

\date{April 22, 2008}

\maketitle

\section{Statement of the Results}

Throughout this paper we let $X$ be a compact metric space.
Furthermore, unless specified otherwise
$f: X \to X$ will be a strictly ergodic homeomorphism (i.e., $f$ is
minimal and uniquely ergodic) that fibers over an almost periodic
dynamical system. This means that there exists an infinite compact
abelian group $\G$ and an onto continuous map $h:X \to \G$
such that $h(f(x))=h(x)+\alpha$ for some $\alpha \in \G$. Examples
of particular interest include:
\begin{itemize}
\item minimal translations of the $d$-torus $\T^d$, for any $d \ge
1$;
\item the skew-shift $(x,y) \mapsto (x + \alpha, y + x)$ on
$\T^2$, where $\alpha$ is irrational.
\end{itemize}

\subsection{Results for $\SL(2,\R)$-Cocycles}

Given a continuous map $A:X \to \SL(2,\R)$, we consider the
skew-product $X \times \SL(2,\R) \to X \times \SL(2,\R)$ given by
$(x,g) \mapsto (f(x), A(x) \cdot g)$. This map is called the
\emph{cocycle} $(f,A)$. For $n\in \Z$, $A^n$ is defined by
$(f,A)^n = (f^n, A^n)$.

We say a cocycle $(f,A)$ is \emph{uniformly
hyperbolic}\footnote{Some authors say that the cocycle has an
\emph{exponential dichotomy}.} if there exist constants $c>0$,
$\lambda>1$ such that $\|A^n(x)\| > c \lambda^n$ for every $x\in
X$ and $n>0$.\footnote{If $A$ is a real $2 \times 2$ matrix, then
$\|A\| = \sup_{\|v\|\neq 0} \|A(v)\|/\|v\|$, where $\|v\|$ is the
Euclidean norm of $v \in \R^2$.} This is equivalent to the usual
hyperbolic splitting condition: see \cite{Yoccoz}. Recall that
uniform hyperbolicity is an open condition in $C^0(X, \SL(2,\R))$.

We say that two cocycles $(f,A)$ and $(f,\tilde A)$ are
\emph{conjugate} if there exists
a \emph{conjugacy} $B \in C^0(X,\SL(2,\R))$ such that
$\tilde{A}(x) = B(f(x))A(x)B(x)^{-1}$.

\smallskip

Our first result is:

\begin{thm}\label{t.bounded}
Let $f$ be as above. If $A:X \to \SL(2,\R)$ is a continuous map such
that the cocycle $(f,A)$ is not uniformly hyperbolic, then there
exists a continuous $\tilde A: X \to \SL(2,\R)$, arbitrarily
$C^0$-close to $A$, such that the cocycle $(f,\tilde A)$ is
conjugate to an $\SO(2,\R)$-valued cocycle.
\end{thm}

\begin{rem}
A cocycle $(f,A)$ is conjugate to a cocycle of rotations if and
only if there exists $C>1$ such that $\|A^n(x)\| \le C$ for every
$x\in X$ and $n\in \Z$ (here it is enough to assume that $f$ is
minimal); see \cite{Cam,EJ,Yoccoz}.
\end{rem}

\begin{rem}\label{rem.minimal suffices}
In Theorem~\ref{t.bounded}, one can drop the hypothesis of unique
ergodicity of $f$ (still asking $f$ to be minimal and to fiber
over an almost periodic dynamics), as long as $X$ is finite dimensional.
See Remark~\ref{rem.minimal explanation}.
\end{rem}

Next we focus on the opposite problem of approximating a cocycle by one that
is uniformly hyperbolic.  As we will see, this problem is
related to the important concept of \emph{reducibility}.

To define reducibility, we will need a slight variation of the notion of
conjugacy.  Let us say that two cocycles $(f,A)$ and $(f,\tilde A)$ are
\emph{$\PSL(2,\R)$-conjugate} if there exists $B \in
C^0(X,\PSL(2,\R))$ such that $\tilde{A}(x) = B(f(x))A(x)B(x)^{-1}$ (the
equality being considered in $\PSL(2,\R)$).  We say that $(f,A)$ is
\emph{reducible} if it is $\PSL(2,\R)$-conjugate to a constant cocycle.

\begin{rem}
Reducibility does not imply, in general, that $(f,A)$ is conjugate to a constant cocycle,
which would correspond to taking $B \in C^0(X,\SL(2,\R))$.
For example, let $X=\T^1$, $f(x) = x+\alpha$.
Let $H =\mathrm{diag}(2,1/2)$, and define $A(x) = R_{- \pi
(x+\alpha)} H R_{\pi x}$.\footnote{$R_\theta$ indicates the
rotation of angle $\theta$.} Notice $A$ is continuous,
$(f, A)$ is $\PSL(2,\R)$-, but not $\SL(2,\R)$-, conjugate to a constant.
(For an example where $(f,A)$ is not uniformly hyperbolic, see
Remark~\ref{rem.example}.)
\end{rem}

Let us say that a $\SL(2,\R)$-cocycle $(f,A)$ is \emph{reducible
up to homotopy} if there exists a reducible cocycle $(f,\tilde A)$
such that the maps $A$ and $\tilde A:X\to\SL(2,\R)$ are homotopic.
Let $\Ruth$ be the set of all $A$ such that $(f,A)$ is reducible
up to homotopy.

\begin{rem}
In the case that $f$ is homotopic to the identity map, it is easy
to see that $\Ruth$ coincides with the set of maps
$A:X\to\SL(2,\R)$ that are homotopic to a constant.
\end{rem}

It is well known that there exists an obstruction to approximating a cocycle
by a uniformly hyperbolic one: a uniformly hyperbolic cocycle is always
reducible up to homotopy (see Lemma \ref {l.uh is ruth}).  Our next result
shows that, up to this obstruction, uniform hyperbolicity is dense.

\begin{thm}\label{t.uh is dense}

Uniform hyperbolicity is dense in $\Ruth$.

\end{thm}

This result is obtained as a consequence of a detailed
investigation of the problem of denseness of reducibility:

\begin{thm}\label{t.reduce}
Reducibility is dense in $\Ruth$.  More precisely:
\begin{enumerate}
\item If $(f,A)$ is uniformly hyperbolic, then it can be
approximated by a reducible cocycle {\rm (}which is uniformly
hyperbolic{\rm )}. \item If $A \in \Ruth$, but $(f,A)$ is not
uniformly hyperbolic and $A_* \in \SL(2,\R)$ is non-hyperbolic
{\rm (}i.e., $\left|\tr A_* \right| \leq 2${\rm )}, then $(f,A)$
lies in the closure of the $\PSL(2,\R)$ conjugacy class of
$(f,A_*)$.
\end{enumerate}
\end{thm}

\begin{proof}[Proof of Theorem~\ref {t.uh is dense}]

The closure of the
set of uniformly hyperbolic cocycles is obviously invariant under
$\PSL(2,\R)$ conjugacies, and clearly
contains all constant cocycles $(f,A_*)$ with $\tr A_*=2$.  The result
follows from the second part of Theorem~\ref {t.reduce}.
\end{proof}

\begin{rem}

It would be interesting to investigate also the closure of an arbitrary
$\PSL(2,\R)$ conjugacy class.  Even the case of the $\PSL(2,\R)$ conjugacy
class of a constant hyperbolic cocycle already escapes our methods.

\end{rem}

Let us say a few words about the proofs and relation with the
literature. In the diffeomorphism and flow settings, Smale
conjectured in the 1960's that hyperbolic dynamical systems are
dense. This turned out to be false in general. However, there are
situations where denseness of hyperbolicity holds; see, for
example, the recent work \cite{KSS} in the context of
one-dimensional dynamics.

Cong~\cite{Cong} proved that uniform hyperbolicity is (open and)
dense in the space of $L^\infty(X, \SL(2,\R))$-cocycles, for any
base dynamics $f$. So our Theorem~\ref{t.uh is dense} can be seen
as a continuous version of his result. Cong's proof involves a
tower argument to perturb the cocycle and produce an invariant
section for its action on the circle $\P^1$. We develop a somewhat
similar technique, replacing $\P^1$ with other spaces. Special
care is needed in order to ensure that perturbations and sections
be continuous.

Another related result was obtained by Fabbri and Johnson who
considered continuous-time systems over translation flows on
$\T^d$ and proved for a generic translation vector that uniform
hyperbolicity occurs for an open and dense set of cocycles; see
\cite{FJ}.

\subsection{Results for Schr\"odinger Cocycles}

We say $(f,A)$ is a \emph{Schr\"odinger cocycle} when $A$ takes its values in the set
\begin{equation}\label{sdefinition}
S = \left\{ \begin{pmatrix} t & -1 \\ 1 & 0 \end{pmatrix} ; \; t \in \R \right\} \subset \SL(2,\R) \, .
\end{equation}
The matrices $A^n$ arising in the iterates of a Schr\"odinger
cocycle are the so-called transfer matrices associated with a
discrete one-dimensional Schr\"odinger operator.

More explicitly, given $V \in C^0(X,\R)$ (called the \emph{potential})
and $x \in X$, we
consider the operator
\begin{equation}\label{oper}
(H_x \psi)_n = \psi_{n+1} + \psi_{n-1} + V(f^n x) \psi_n
\end{equation}
in $\ell^2(\Z)$. Notice that $u$ solves the difference equation
\begin{equation}\label{eve}
u_{n+1} + u_{n-1} + V(f^n x) u_n = E u_n
\end{equation}
if and only if
\begin{equation}\label{evemat}
\begin{pmatrix} u_n \\ u_{n-1} \end{pmatrix} = A^n_{E,V} \begin{pmatrix} u_0 \\
u_{-1}
\end{pmatrix},
\end{equation}
where $(f,A_{E,V})$ is the Schr\"odinger cocycle with
$$
A_{E,V}(x) = \begin{pmatrix} E - V(x) & - 1 \\ 1 & 0
\end{pmatrix}.
$$
Properties of the spectrum and the spectral measures of the
operator \eqref{oper} can be studied by looking at the solutions
to \eqref{eve} and hence, by virtue of \eqref{evemat}, the
one-parameter family of Schr\"odinger cocycles $(f,A_{E,V})$.
Using minimality of $f$, it follows quickly by strong operator
convergence that the spectrum of $H_x$ is independent of $x \in X$
and we may therefore denote it by $\Sigma$. It is well known that
$\Sigma$ is a perfect set; as a spectrum it is closed and there
are no isolated points by ergodicity of $f$ and
finite-dimensionality of the solution space of \eqref{eve} for
fixed $E$. Johnson \cite{J} (see also Lenz~\cite{L}) showed that
$\Sigma$ consists of those energies $E$, for which $(f,A_{E,V})$
is not uniformly hyperbolic:
\begin{equation}\label{jlresult}
\R \setminus \Sigma = \{ E \in \R ; (f,A_{E,V}) \text{ is
uniformly hyperbolic} \}.
\end{equation}

Our results have natural versions for Schr\"odinger cocycles, with
the added simplification that all such cocycles are homotopic to a
constant. Simple repetition of the proofs leads to difficulties in
the construction of perturbations (since there are fewer
parameters to vary). We prove instead a general reduction result,
which is of independent interest. Recall the definition
\eqref{sdefinition} of the set $S$.

\begin{thm}\label{t.P}
Let $f:X \to X$ be a minimal homeomorphism of a compact metric space.
Let $P$ be any conjugacy-invariant property of $\SL(2,\R)$-valued
cocycles over $f$.
If $A \in C^0(X,S)$ can be approximated by
$\SL(2,\R)$-valued cocycles with property $P$, then
$A$ can be approximated by $S$-valued cocycles with property~$P$.
\end{thm}

We are even able to treat the case of more regular cocycles.

\begin{thm} \label {t.regular}
Let $1 \leq r \leq \infty$ and let $0 \leq s \leq r$.
Let $X$ be a $C^r$ compact manifold and let
$f:X \to X$ be a minimal $C^r$-diffeomorphism.
Let $P$ be any property of $\SL(2,\R)$-valued
$C^r$ cocycles over $f$ which is invariant by $C^r$-conjugacy.
If $A \in C^s(X,S)$ can be $C^s$-approximated by
$\SL(2,\R)$-valued cocycles with property $P$, then
$A$ can be $C^s$-approximated by $S$-valued cocycles with property~$P$.
\end{thm}

\begin{rem}
Having in mind applications to other types of difference equations,
it would be interesting to investigate the validity of the above results
for more general classes of sets.
\end{rem}

It follows from Theorems~\ref{t.uh is dense} and \ref{t.P} that
uniformly hyperbolic Schr\"odinger cocycles are $C^0$-dense.  This has
the following corollary:

\begin{corol}\label{c.cantor}
For a generic $V \in
C^0(X,\R)$, we have that $\R \setminus \Sigma$ is dense. That is,
the associated Schr\"odinger operators have Cantor spectrum.
\end{corol}

\begin{proof}
For $E \in \R$, consider the set
$$
UH_E = \big\{ V \in C^0(X,\R); (f,A_{E,V}) \text{ is uniformly
hyperbolic} \big\}.
$$
By Theorems~\ref{t.uh is dense} and \ref{t.P}, $UH_E$ is (open
and) dense. Thus, we may choose a countable dense subset $\{E_n\}$
of $\R$ and then use \eqref{jlresult} to conclude that for $V \in
\bigcap_n UH_{E_n}$, the set $\R \setminus \Sigma$ is dense.
\end{proof}

Let us discuss this result in the two particular cases of
interest, translations and skew-shifts on the torus.

If the base dynamics is given by a translation on the torus, that
is, for quasi-periodic Schr\"odinger operators, Cantor spectrum is
widely expected to occur generically. Corollary~\ref{c.cantor}
proves this statement in the $C^0$-topology. There are other
related results that also establish a genericity statement of this
kind. Cong and Fabbri~\cite{CF} consider bounded measurable
potentials $V$. Fabbri, Johnson, and Pavani studied quasi-periodic
Schr\"odinger operators in the continuum case, that is, acting in
$L^2(\R)$. They prove for generic translation vector that Cantor
spectrum is $C^0$-generic; see \cite{FJP}. More recently, generic
Cantor spectrum for almost periodic Schr\"odinger operators in the
continuum was established by Gordon and Jitomirskaya \cite {GJ}.

On the other hand, Corollary~\ref{c.cantor} is rather surprising
in the case of the skew-shift. Though few results are known, it is
often assumed that in many respects the skew-shift behaves
similarly to a Bernoulli shift, and Schr\"odinger operators
associated to Bernoulli shifts never have Cantor spectrum. More
precisely, the following is expected for Schr\"odinger operators
defined by the skew-shift and a sufficiently regular non-constant
potential function $V : \T^2 \to \R$; compare \cite[p.~114]{Bou3}.
The (top) Lyapunov exponent of $(f,A_{E,V})$ is strictly positive
for almost every $E \in \R$, the operator $H_{(x,y)}$ has pure
point spectrum with exponentially decaying eigenfunctions for
almost every $(x,y) \in \T^2$, and the spectrum $\Sigma$ is not a
Cantor set. Some partial affirmative results concerning the first
two statements can be found in \cite{Bou1,Bou2,Bou3,BGS}, whereas
Corollary~\ref{c.cantor} above shows that the third expected
property fails generically in the $C^0$ category.\footnote{A
result of a similar flavor was recently obtained in \cite{BD}; if
$\alpha$ is not badly approximable, then the second expected
property also fails generically in the $C^0$
category.}
It is natural to pose the question of whether our result is an
artifact of weak regularity: can the spectrum of a Schr\"odinger
operator associated to the skew-shift with analytic potential ever
be a Cantor set?

The following result will follow quickly from the results
described above and the standard KAM theorem:

\begin{thm} \label {t.acspectrum}
Assume that $f$ is a Diophantine translation of the $d$-torus.
Then the set of $V \in C^0(X,\R)$ for which the corresponding
Schr\"odinger operators have some absolutely continuous spectrum
is dense.
\end{thm}

This should be compared with \cite{AD}, which showed that singular
spectrum is $C^0$-generic in the more general context of ergodic
Schr\"odinger operators.

\section{Proof of the Results for $\SL(2,\R)$-Cocycles}
\label{s.theorems123}

Our goal is to prove Theorems~\ref{t.bounded} and \ref{t.reduce}.

\begin{lemma}\label{l.group dichotomy}
There are two possibilities about the group $\G$:
\begin{enumerate}
\item \emph{(the circle case)} either there is an onto continuous
homomorphism $s: \G \to \T^1$; \item \emph{(the Cantor case)} or $\G$ is a
Cantor set.
\end{enumerate}
In the second alternative, there exist continuous homomorphisms from $\G$
onto finite cyclic groups of arbitrarily large order.
\end{lemma}

In the lack of an exact reference, a proof of Lemma~\ref{l.group
dichotomy} is given in Appendix~\ref{ap.groups}.

We will first work out the arguments for the more difficult
circle case.
By assumption, $f$ fibers over a translation on $\G$, and hence also
over a translation of the circle.
That translation is minimal, because so is $f$.
Therefore \emph{in the circle case, we can and we do assume that $\G = \T^1$.}

The proofs then go as follows:
In \S\ref{ss.almost invariant} we explain a construction of
almost-invariant sections for skew-products. This is used in
\S\ref{ss.real} to prove Lemma~\ref{l.real}, which says that
functions that are cohomologous to constant are dense in
$C^0(X,\R)$. Using Lemma~\ref{l.real}, the first case of
Theorem~\ref{t.reduce} is easily proven in \S\ref{ss.reduce uh}.
In \S\ref{ss.disk} we establish some lemmas about the action of
$\SL(2,\R)$ on hyperbolic space. The proof of
Theorem~\ref{t.bounded} is given in \S\ref{ss.bounded}. It is
similar to the proof of Lemma~\ref{l.real}, with additional
ingredients: results on Lyapunov exponents from \cite{Bochi} and
\cite{Furman} and the material from \S\ref{ss.disk}. To prove the
second case of Theorem~\ref{t.reduce} in \S\ref{ss.final} we
employ Theorem~\ref{t.bounded} and Lemma~\ref{l.real} again.

In \S\ref{ss.cantor groups} we will discuss the Cantor case,
which is obtained by a simplification of the arguments
(because then no gluing considerations are needed).

\subsection{Almost-Invariant Sections}\label{ss.almost invariant}
\emph{From here until \S\ref{ss.final} we consider only the circle case.}

A continuous \emph{skew-product} over $f$ is a continuous map $F:X
\times Y \to X \times Y$ (where $Y$ is some topological space) of
the form $F(x,y) = (f(x), F_x(y))$. $F$ is called
\emph{invertible} if it is a homeomorphism of $X \times Y$. In
this case, we write $F^n(x,y) = (f^n(x), F^n_x(y))$, for $n\in
\Z$. An \emph{invariant section} for $F$ is a continuous map $x
\mapsto y(x)$ whose graph is $F$-invariant.

\medskip

Let $p_i/q_i$ be the $i$th continued fraction approximation of
$\alpha$.
We recall that $q_i \alpha$ is closer to $0$ than any $n \alpha$ with $1 \le n < q_i$; moreover
the points $q_i \alpha$ alternate sides around $0$.

Let $I_i\subset \T^1$ be the shortest closed interval
containing $0$ and $q_i \alpha$.
Notice that the first $n>0$ for which $I_i + n\alpha$ intersects $I_i$ is $n=q_{i+1}$.
Moreover, $(I_i + n\alpha) \setminus I_i$ coincides (modulo a point) with $I_{i+1}$.
Also notice that $I_{i+1} + q_i \alpha$ is contained in $I_i$.

Let $i$ be fixed.
The above remarks show that the family of intervals
\begin{equation}\label{e.family}
I_i, \ I_i + \alpha, \ldots, I_i + (q_{i+1}-1) \alpha, \
I_{i+1},\  I_{i+1}+\alpha, \ldots, I_{i+1} + (q_i - 1) \alpha
\end{equation}
has the following properties:
\begin{itemize}
\item the union of the intervals is the whole circle;
\item the interiors of the intervals are two-by-two disjoint.
\end{itemize}
(Another way to the obtain the family \eqref{e.family}
is to cut the circle along the points $n\alpha$ with $0 \le n \le q_{i+1} + q_i - 1$.)
We draw the intervals from~\eqref{e.family} from bottom to top as in Figure~\ref{f.tower}.
Then each point is mapped by the $\alpha$-rotation
to the point directly above it, or else to somewhere in the bottom floor $I_i$.
\psfrag{1}[][c]{{\tiny $0$}}                  
\psfrag{2}[][c]{{\tiny $(q_{i+1}+q_i)\alpha$}}
\psfrag{3}[][c]{{\tiny $q_i\alpha$}}
\psfrag{5}[][c]{{\tiny $(q_i-1)\alpha$}}
\psfrag{6}[][c]{{\tiny $-\alpha$}}
\psfrag{7}[][c]{{\tiny $(q_{i+1}+q_i-1)\alpha$}}
\psfrag{a}[][l]{{\tiny $q_i$}}               
\psfrag{b}[][l]{{\tiny $q_{i+1}$}}
\begin{figure}[hbt]
\includegraphics[height=7.5cm]{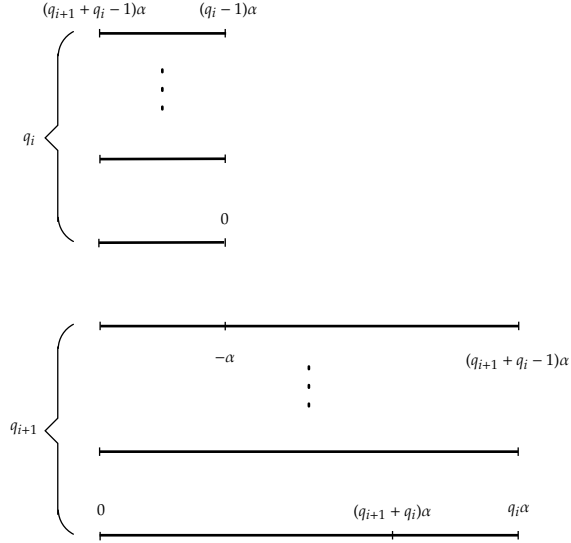}
\caption{Castle with base $I_i$.}
\label{f.tower} 
\end{figure}

\medskip

The following lemma (and its proof) will be used
in several situations (namely, \S\S\ref{ss.real} and \ref{ss.bounded}).

\begin{lemma}\label{l.almost invariant}
Let $F:X \times Y \to X \times Y$ be a continuous invertible
skew-product over $f$. Fix any $i \in \N$, and let $I = I_i$.
Given any map $y_0: h^{-1}(I) \to Y$, there exists a unique map
$y_1: X \to Y$ that extends $y_0$ and such that
\begin{equation}\label{e.almost invariant}
F(x, y_1(x)) = (f(x), y_1(f(x))) \quad \text{for all $x \in X \setminus h^{-1}(I)$.}
\end{equation}
If, in addition, $y_0$ is continuous and satisfies
\begin{equation}\label{e.compatible}
F^n (x, y_0(x)) = (f^n(x), y_0( f^n(x))) \text{ for all $x \in
h^{-1}(0)$, $n \in \{q_i, q_{i+1}+q_i\}$},
\end{equation}
then $y_1$ is continuous.
\end{lemma}

\begin{proof}
For each $x \in X$, let
\begin{equation}\label{e.def tau}
\tau (x) = \min \, \{n \ge 0; \; f^n(x) \in h^{-1}(I) \} \, .
\end{equation}
Given $y_0: h^{-1}(I) \to Y$, then $y_1$ must be given by
\begin{equation}\label{e.defn y1}
y_1(x) = \left(F^{\tau(x)}_x\right)^{-1} \left( y_0(f^{\tau(x)}(x)) \right) \, .
\end{equation}

Now assume that $y_0$ is continuous and \eqref{e.compatible} holds.
We only need to check that $y$ is
continuous at each point $x$ where $\tau$ is not.
Fix such an $x$ and let $k = \tau(x)$.
Then either (i) $f^k(x) \in h^{-1}(0)$
or (ii) $f^k(x) \in h^{-1}(q_i \alpha)$.
Let $\ell = \tau(f^k(x))$, that is,
$\ell = q_i$ in case (i), and $\ell = q_{i+1}$ in case (ii).
Due to the definition of $y_0$, we have
$F^{\ell}_{f^k(x)} \big(y_0(f^k(x))\big) = y_0(f^{k+ \ell}(x))$
in both cases.
Therefore
$$
\left( F^{k}_{x} \right)^{-1} \left( y_0(f^{k}(x)) \right) =
\left( F^{k+\ell}_{x} \right)^{-1} \left( y_0(f^{k+\ell}(x)) \right).
$$
The set of the possible limits of $\tau(x_j)$, where $x_j \to x$,
is precisely $\{k, k+\ell\}$.
It follows that $y$ is continuous at $x$.
\end{proof}

\subsection{The Cohomological Equation} \label{ss.real}

\begin{lemma}\label{l.real}
For every $\phi \in C^0(X,\R)$ and every $\delta>0$, there exists
$\tilde \phi \in C^0(X,\R)$ such that $\|\phi-\tilde\phi\|_{C^0} <
\delta$ and $\tilde \phi$ is $C^0$ cohomologous to a constant:
there exist $w \in C^0(X,\R)$ and $a_0 \in \R$ such that
$$
\tilde \phi = w \circ f - w + a_0 \, .
$$
\end{lemma}

\begin{rem}
In the case $X=\T^1$, there is a quick proof of
Lemma~\ref{l.real}: Approximate $\phi$ by a (real) trigonometric
polynomial $\tilde \phi(z) = \sum_{|n|\le m} a_n z^n$, and let
$w(z) = \sum_{0<|n|\le m} (e^{2\pi i n \alpha} - 1)^{-1} a_n z^n$.
\end{rem}

The following proof contains a construction that will
appear again in the (harder) proof of Theorem~\ref{t.bounded},
so it may also be useful as a warm-up.

\begin{proof}[Proof of Lemma~\ref{l.real}]
Fix $\phi \in C^0(X,\R)$ and $\delta>0$ small.
Let $a_0$ be the integral of $\phi$ with respect to the unique $f$-invariant probability measure.
Without loss of generality, assume $a_0=0$.
Write $S_n = \sum_{j=0}^{n-1} \phi \circ f^j$.
Let $n_0$ be such that $\left| S_n / n \right| < \delta$
uniformly for every $n \ge n_0$.

Choose and fix $i$ such that
$$
q_i > \max \left(n_0, \delta^{-1} \|\phi\|_{C^0}\right).
$$
Let $I = I_i$. The rest of the proof is divided into three steps:

\medskip
\noindent \textit{Step 1: Finding an almost-invariant section $w_1: X \to \R$.}
First define a real function $w_0$ on $h^{-1}(\{0, q_i\alpha, (q_{i+1}+q_i)\alpha\})$ by
$$
w_0(f^n (x)) = S_n (x)
\text{ for $x \in h^{-1}(0)$ and $n=0$, $q_i$, or $q_{i+1}+q_i$.}
$$
Using Tietze's Extension Theorem, we extend continuously $w_0$ to
$h^{-1}(I)$ so that
$$
\sup_{h^{-1}(I)} |w_0|=
\sup_{h^{-1} (\{0, q_i\alpha, (q_{i+1}+q_i)\alpha\})} |w_0| .
$$
Now we consider the skew-product
$$
F : X \times \R \to X \times \R , \; F(x,w) = (f(x), w + \phi(x)).
$$
Applying Lemma~\ref{l.almost invariant} to $F$ and $w_0$, we find
a continuous function $w_1:X \to \R$ which extends $w_0$ and such
that $w_1(f(x)) = w_1(x) + \phi(x)$ if $x \not\in h^{-1}(\interior
I)$.

\medskip
\noindent \textit{Step 2: Definition of functions $\tilde \phi$, $w: X \to \R$.}
Define $\tilde \phi$ by $\tilde \phi = \phi$ outside of
$\bigsqcup_{n=0}^{q_{i+1}-1} f^n(h^{-1}(I))$, and
$$
\tilde\phi (f^n(x)) = \phi(f^n(x)) +
\frac{w_1(f^{q_{i+1}}(x)) - w_1(x) - S_{q_{i+1}}(x)}{q_{i+1}} \quad
\text{if $x\in h^{-1}(I)$, $0 \le n < q_{i+1}$.}
$$
Define $w$ by $w=w_1$ outside of $\bigsqcup_{n=1}^{q_{i+1}-1} f^n(h^{-1}(I))$, and
$$
w (f^n (x)) = w(x) + \sum_{j=0}^{n-1}\tilde\phi(f^j(x))
\quad \text{if $x\in h^{-1}(I)$, $0 \le n < q_{i+1}$.}
$$
Then $\tilde\phi$ and $w$ are continuous functions satisfying
$\tilde \phi = w \circ f - w$.

\medskip
\noindent \textit{Step 3: Distance estimate.}
From now on, let $x \in h^{-1}(I)$ be fixed.
Due to the definition of $w_0$ we have
$$
|w_0(x)| \le (q_{i+1} + q_i)\delta.
$$
Recalling~\eqref{e.def tau}, we see that $\tau (f^{q_{i+1}}(x))$
equals either $1$ or $q_i+1$ (see Figure~\ref{f.tower}). In any case,
$|S_{\tau(x)}(x)| \le (q_i+1) \delta$ and therefore,
by~\eqref{e.defn y1},
$$
|w_1(f^{q_{i+1}}(x))| \le |w_0(f^{\tau(x)+q_{i+1}}(x))| + |S_{\tau(x)}(x)|
\le (q_{i+1} + 2q_i + 1)\delta.
$$
Hence
\begin{align*}
\left| \frac{w_1(f^{q_{i+1}}(x)) - w_1(x)}{q_{i+1}} \right| \le
\frac{(3q_{i+1}+3q_i+1)\delta}{q_{i+1}} < 7\delta.
\end{align*}
That is, the $C^0$ distance between $\tilde\phi$ and $\phi$ is $<7\delta$.
\end{proof}

\subsection{Denseness of Reducibility in the Uniformly Hyperbolic Case}
\label{ss.reduce uh}

First, let us note a basic fact:
\begin{lemma}\label{l.uh is ruth}
For any homeomorphism $f:X \to X$, if $(f,A)$ is uniformly hyperbolic, then $A \in \Ruth$.
\end{lemma}

\begin{proof}
By uniform hyperbolicity, for each $x\in X$
there exists a splitting $\R^2= E^u(x) \oplus E^s(x)$, which depends continuously on $x$
and is left invariant by the cocycle, that is,
$A(x) \cdot E^{u,s}(x) = E^{u,s}(f(x))$.

Let $\{e_1,e_2\}$ be the canonical basis of $\R^2$.
For each $x \in X$, let $e^u(x) \in E^u(x)$ and $e^s(x) \in E^s(x)$
be unit vectors so that $\{e^u(x), e^s(x)\}$ is a positively oriented basis.
Define a matrix $B(x)$ putting $B(x) \cdot e_1 = ce^u(x)$ and $B(x) \cdot e_2 = ce^s(x)$,
where $c = [\sin \sphericalangle(e^u(x), e^s(x))]^{-1/2}$ is chosen so that $\det B(x)=1$.
Then $B(x)$ is uniquely defined as an element of $\PSL(2,\R)$, and depends continuously on $x$.

Let $D(x)$ be given by $A(x) = B(f(x)) D(x) B(x)^{-1}$.
Then $D(x)$ is a diagonal ``matrix''.
Therefore $D:X \to \PSL(2,\R)$ is homotopic to a constant
and $A$ is homotopic to a reducible cocycle.
\end{proof}

\begin{proof}[Proof of the first part of Theorem~\ref{t.reduce}]
Let us write, for $t \in \R$,
$
D_t = \pm \begin{pmatrix}
e^t & 0 \\ 0 & e^{-t}
\end{pmatrix} \in \PSL(2,\R).
$ By the proof of Lemma~\ref{l.uh is ruth}, there exist $B \in
C^0(X, \PSL(2,\R))$  and $\phi \in C^0(X,\R)$ such that $A(x) =
B(f(x)) D_{\phi(x)} B(x)^{-1}$. By Lemma~\ref{l.real}, we can
perturb $\phi$ (and hence $A$) in the $C^0$-topology so that $\phi
= w \circ f - w + a_0$ for some $w\in C^0(X,\R)$ and $a_0 \in \R$.
We can assume $a_0 \neq 0$. Then $\hat{B}(x) = B(x) D_{w(x)}$ is a
conjugacy between $A$ and the constant $D_{a_0}$.
\end{proof}

\subsection{Disk Adjustment Lemma} \label{ss.disk}
The aim here is to establish Lemma~\ref{l.adjust} below, that will be used in the proof of Theorem~\ref{t.bounded}.
First we need to recall some facts about hyperbolic geometry.

The group $\SL(2,\R)$ acts on the upper half-plane
$\H =\{w\in \C;\; {\mathrm{Im}\, w>0}\}$
as follows:
$$
A = \begin{pmatrix} a & b \\ c & d \end{pmatrix} \in \SL(2,\R) \ \Rightarrow \
A \cdot w = \frac{aw+b}{cw+d} \, .
$$
(In fact, the action factors through $\PSL(2,\R)$.)
We endow the half-plane with the Riemannian metric (of curvature $-1$)
$$
v \in T_w \H \ \Rightarrow \  \|v\|_w = \frac{|v|}{\mathrm{Im} \, w} \, .
$$
Then $\SL(2,\R)$ acts on $\H$ by isometries.

We fix the following conformal equivalence between $\H$ and the unit disk
$\D = \{z \in \C ; \; |z|<1\}$:
$$
w = \frac{-iz-i}{z-1} \in \H \ \leftrightarrow \ z =
\frac{w-i}{w+i} \in \D \, .
$$
We take on the disk the Riemannian metric that makes the map above an isometry, namely
$\|v\|_z = 2(1-|z|^2)^{-2} |v|$.
By conjugating, we get an action of $\SL(2,\R)$
on $\D$ by isometries, that we also denote as $(A,z) \mapsto A\cdot z$.

Let $d(\mathord{\cdot}, \mathord{\cdot})$ denote the distance function
induced on $\D$ by the Riemannian metric.
We claim that:
\begin{equation}\label{e.norm}
\|A\| = e^{d(A \cdot 0,0)/2} \quad \text{for all $A \in \SL(2,\R)$.}
\end{equation}

\begin{proof}
It is sufficient to prove the corresponding fact $\|A\| = e^{d(A \cdot i,i)/2}$ on the half-plane $\H$.
We first check the case where $A$ is a diagonal matrix
$H_\lambda = \begin{pmatrix} \lambda & 0 \\ 0 & \lambda^{-1} \end{pmatrix}$
with $\lambda >1$:
$$
d(A \cdot i , i) = d (\lambda^2 i, i) = \int_1^{\lambda^2} \frac{dy}{y} = 2 \log \lambda = 2 \log \|A\| \, .
$$
Next, if $A$ is a rotation $R_\theta$ then its action on $\H$ fixes the point $i$, so the claim also holds.
Finally, a general matrix can be written as $A = R_\beta H_\lambda R_\alpha$,
so~\eqref{e.norm} follows.
\end{proof}


We now prove two lemmas.

\begin{lemma}\label{l.first adjust}
There exists a continuous map $\Phi:\D \times \D \to \SL(2,\R)$ such that
$\Phi(p_1,p_2) \cdot p_1 = p_2$ and
$\|\Phi(p_1,p_2)- \Id\| \leq e^{d(p_1,p_2)/2}-1$.
\end{lemma}

Let us recall a few more facts about the half-plane and disk models, that we will use in the proof of the lemma.
An \emph{extended circle} means either an Euclidean circle or an Euclidean line in the complex plane.
\begin{itemize}
\item The geodesics on $\H$ (resp.~$\D$) are arcs of extended circles that meet orthogonally the real axis $\partial \H$ (resp.~the unit circle $\partial \D$) at the endpoints (called \emph{points at infinity}).
\item The points lying at a fixed positive distance from a geodesic $\gamma$ form two arcs of extended circles $\gamma_1$ and $\gamma_2$ that have the same points at infinity as $\gamma$. See Figure~\ref{f.facts}~(left).
    We say that $\gamma$ and $\gamma_1$ are \emph{equidistant curves}.
\item A quadrilateral $p_1 q_1 q_2 p_2$ is called a \emph{Saccheri quadrilateral} with \emph{base} $q_1 q_2$ and \emph{summit} $p_1 p_2$ if the angles at vertices $q_1$ and $q_2$ are right and the sides $p_1 q_1$ and $p_2 q_2$ (called the \emph{legs}) have the same length. See Figure~\ref{f.facts}~(right). The fact is that the summit is necessarily longer than the base.
\end{itemize}

\psfrag{g}[][r]{{\tiny $\gamma$}}
\psfrag{g1}[][l]{{\tiny $\gamma_1$}}
\psfrag{g2}[][r]{{\tiny $\gamma_2$}}
\psfrag{p1}[][l]{{\tiny $p_1$}} \psfrag{p2}[][r]{{\tiny $p_2$}}
\psfrag{q1}[][l]{{\tiny $q_1$}} \psfrag{q2}[][r]{{\tiny $q_2$}}
\begin{figure}[hbt]
\includegraphics[height=3.5cm]{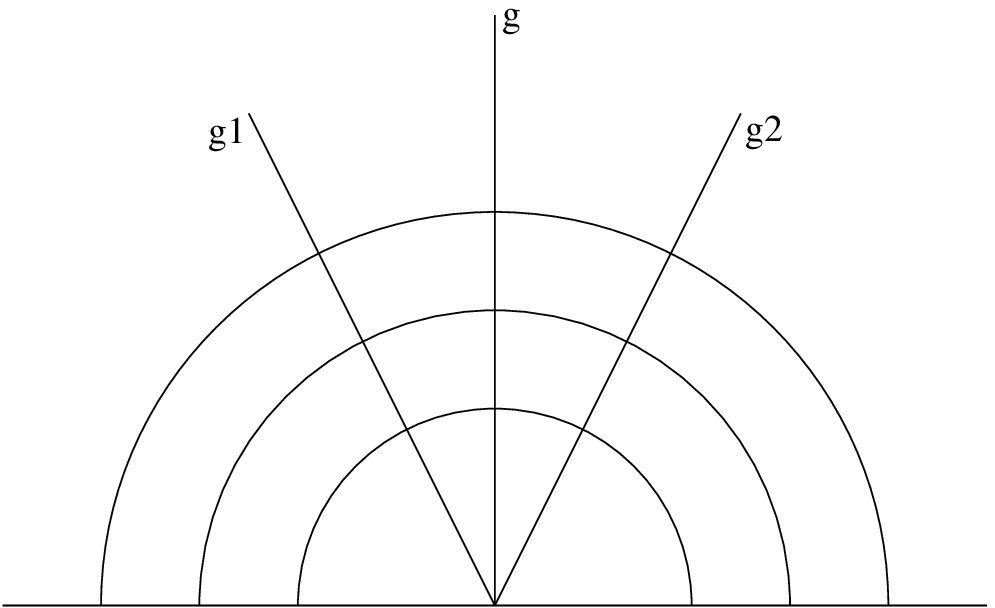}
\qquad\qquad
\includegraphics[height=4.5cm]{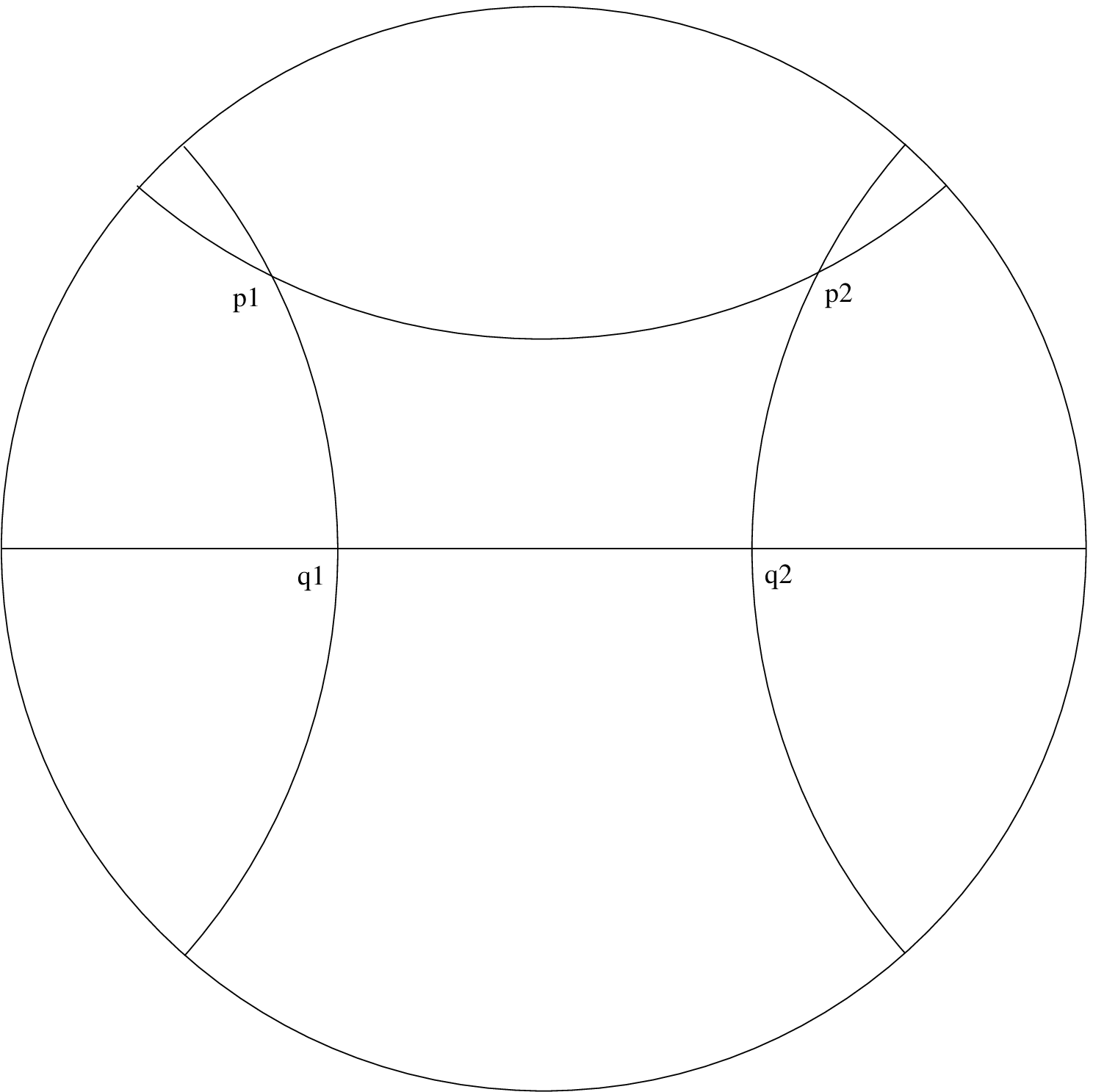}
\caption{Left: equidistant curves on $\H$.
Right: a Saccheri quadrilateral $p_1 q_1 q_2 p_2$ on $\D$.}
\label{f.facts}
\end{figure}

\begin{proof}[Proof of Lemma~\ref{l.first adjust}]
We will define the matrix $\Phi(p_1, p_2)$ explicitly.
It is the identity if $p_1 = p_2$, so from now on consider $p_1 \neq p_2$.

We first consider a particular case,
where we rewrite the two points as $q_1$, $q_2$.
Assume that the (whole) geodesic $\gamma$ containing $q_1$ and $q_2$
also contains $0$.
That is, $\gamma$ is a piece of Euclidean line.
Let $u$ be the point in the circle $\{|z|=1\}$
such that the line contains $-u$, $q_1$, $q_2$, $u$, in this order.
Consider the hyperbolic isometry
that preserves the geodesic $\gamma$, translating it and taking $q_1$ to $q_2$.
That isometry corresponds to a matrix of the form:
$$
A = R_{\theta} H_\lambda R_{-\theta} \, ,
\quad \text{where }
R_\theta = \begin{pmatrix} \cos \theta & -\sin \theta \\ \sin\theta & \cos \theta \end{pmatrix}, \
H_\lambda = \begin{pmatrix} \lambda & 0 \\ 0 & \lambda^{-1} \end{pmatrix},
$$
for some $\theta \in \R$ (in fact, $e^{-2i\theta} = u$) and $\lambda>1$.
Since the isometry translates $\gamma$, we have
$d(0, A \cdot 0) = d(q_1, A \cdot q_1) = d(q_1, q_2)$.
Therefore~\eqref{e.norm} gives
$\lambda = \|A\| =  e^{d(q_1,q_2)/2}$.
On the other hand, $\|A - \Id \| = \|H_\lambda - \Id\| = \lambda - 1$,
so we can define $\Phi(q_1,q_2) = A$ and the bound claimed in the statement of the lemma
becomes an equality.

\psfrag{mu}[][l]{{\tiny $-u$}}\psfrag{u}[][r]{{\tiny $u$}} \psfrag{p1}[][l]{{\tiny $p_1$}}
\psfrag{p2}[][cu]{{\tiny $p_2$}} \psfrag{q1}[][l]{{\tiny $q_1$}}
\psfrag{q2}[][r]{{\tiny $q_2$}} \psfrag{0}[][c]{{\tiny $0$}}
\psfrag{g}{{\tiny $\gamma$}} \psfrag{tg}{{\tiny $\tilde \gamma$}}
\begin{figure}[hbt]
\includegraphics[height=6cm]{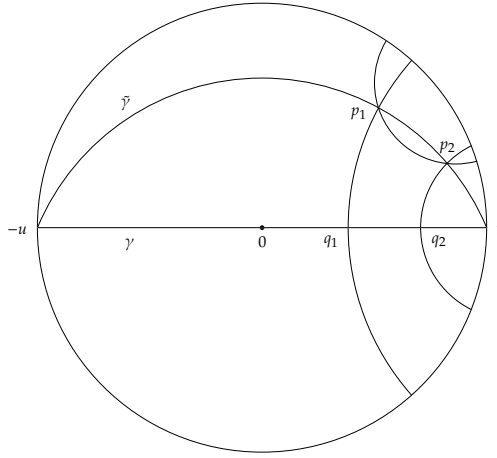}
\caption{Proof of Lemma~\ref{l.first adjust}.}
\label{f.proof}
\end{figure}

Next let us consider the general case.
Given $p_1$ and $p_2$,
consider the family of extended circles that contain $p_1$ and $p_2$.
There exists a unique $C$ in this family that intersects the circle $\{|z|=1\}$
in two antipodal points $u$ and $-u$.
(See Figure~\ref{f.proof}).
Let $\tilde \gamma = C \cap \D$, and let
$\gamma$ be the geodesic whose points at infinity are $u$ and $-u$;
so $\gamma$ and $\tilde \gamma$ are equidistant curves.
Notice that the case already treated corresponds to the case where $\tilde \gamma = \gamma$
is a geodesic.
Let $q_1$, resp.~$q_2$, be the point in $\gamma$ which has the least hyperbolic distance to
$p_1$, resp.~$p_2$.
Notice that $\Phi(q_1,q_2)$ is already defined.
We set  $\Phi(p_1, p_2) = \Phi(q_1, q_2)$.

Because $\gamma$ and $\tilde \gamma$ are equidistant, we have $d(p_1,q_1) = d(p_2, q_2)$.
It follows that $p_1 q_1 q_2 p_2$ is a Saccheri quadrilateral.
In particular, $d(q_1, q_2) < d(p_1, p_2)$ and hence $\|\Phi(p_1,p_2)\| < e^{d(p_1,p_2)/2}$.
Also, since $\Phi(q_1,q_2)$ translates the geodesic $\gamma$ sending $q_1$ to $q_2$, it
sends the leg $q_1 p_1$ to the leg $q_2 p_2$, and in particular, sends $p_1$ to $p_2$,
as desired.

This completes the definition of $\Phi$; continuity is evident.
\end{proof}

\begin{lemma}\label{l.adjust}
For every $n \geq 1$, there exists a continuous map
$\Psi_n : \SL(2,\R)^n \times \D^2 \to \SL(2,\R)^n$ such that
if $\Psi_n(A_1,\ldots,A_n,p,q) = (\tilde A_1,\ldots,\tilde A_n)$ then
\begin{enumerate}
\item $\tilde A_n \cdots \tilde A_1 \cdot p = q$ and
\item $\|\tilde A_i A_i^{-1} - \Id\| \leq
\exp \left( \tfrac{1}{2n} d(A_n \cdots A_1 \cdot p,q)) \right)-1$ for $1 \leq i
\leq n$.
\end{enumerate}
\end{lemma}

\begin{proof}
Let $w_0 = A_n \cdots A_1 \cdot p$ and $L = d(w_0,q)$. For $1 \leq
i \leq n$, let $w_i$ be the point in the hyperbolic geodesic
segment joining $w_0$ and $q$ which is at distance $\frac {iL}{n}$
of $w_0$. Let also $z_i = (A_n \cdots A_{n-i+1})^{-1} \cdot w_i$,
for $0 \leq i \leq n$. Then $z_0=p$, $z_n=q$ and $d(A_i \cdot
z_{i-1}, z_i) =\frac {L}{n}$. Let $\tilde A_i = \Phi(A_i \cdot
z_{i-1}, z_i) A_i$, where $\Phi$ is as in Lemma~\ref{l.first
adjust}.
\end{proof}

\subsection{Proof of Theorem~\ref{t.bounded}} \label{ss.bounded}

\begin{proof}
Let $A: X \to \SL(2,\R)$ be such that $(f,A)$ is not uniformly
hyperbolic, and let $\delta_0>0$ be given. We want to find $\tilde
A \in C^0(X,\SL(2,\R))$ with $\|\tilde A-A\|_{C^0}<\delta_0$, and
a continuous function $z:X \to \D$ such that $\tilde A(x) \cdot
z(x)=z(f(x))$. Accomplishing this, we simply set $B(x) =
\Phi(z(x),0)$ (where $\Phi$ is given by Lemma~\ref{l.first
adjust}) and then $B(f(x)) A(x) B(x)^{-1}$ will be rotations.

Because the cocycle is not uniformly hyperbolic,
a theorem by Bochi~\cite{Bochi} gives a $C^0$-perturbation of $A$ whose
upper Lyapunov exponent (with respect to the unique invariant probability) is zero.
For simplicity of writing, let $A$ denote this perturbation.
Since $f$ is uniquely ergodic, a result due to Furman~\cite{Furman} gives that
$(f,A)$ has \emph{uniform subexponential growth}, that is,
\begin{equation}\label{e.unif subexp growth}
\lim_{n \to \infty} \frac{1}{n} \log \|A^n(x)\|= 0 \quad \text{uniformly on $x\in X$.}
\end{equation}

Let $\delta>0$ be such that $(e^{7 \delta}-1)\|A\|_{C^0} < \delta_0$.
Let $n_0$ be such that
$n \ge n_0$ implies $\|A^n(x)\| \le e^{n\delta}$ for every $x$.

The rest of the argument is analogous to the corresponding steps
in the proof of Lemma~\ref{l.real}, with the disk playing the role
of the line. Let $p_i/q_i$ be the $i$th continued fraction
approximation of $\alpha$. Choose and fix $i$ large so that
$$
q_i > \max \left(n_0, \delta^{-1} \log \|A\|_{C^0}\right).
$$
Let $I\subset \T^1$ be the shortest closed interval containing
$0$ and $q_i \alpha$.
The rest of the proof will be divided into three steps:

\medskip
\noindent \textit{Step 1: Finding an almost-invariant section $z_1: X \to \R$.}
First we define $z_0$ on $h^{-1}(\{0, q_i\alpha, (q_{i+1}+q_i)\alpha\})$ by
$$
z_0(f^n (x)) = A^n (x) \cdot 0
\text{ for $x \in h^{-1}(0)$ and $n=0$, $q_i$, or $q_{i+1}+q_i$.}
$$
Then we extend continuously $z_0$ to $h^{-1}(I)$ in a way such that
$$
\sup_{x\in h^{-1}(I)} d(z_0(x), 0) =
\sup_{x\in h^{-1} (\{0, q_i\alpha, (q_{i+1}+q_i)\alpha\})} d(z_0(x), 0).
$$
Consider the skew-product
$$
F : X \times \D \to X \times \D, \; F(x,z) = (f(x), A(x) \cdot z).
$$
Applying Lemma~\ref{l.almost invariant} to $F$ and $z_0$, we find
a continuous map $z_1:X \to \D$ which extends $z_0$ and such that
$z_1(f(x)) = A(x) \cdot z_1(x) $ if $x \not\in h^{-1}(\interior
I)$.

\medskip
\noindent \textit{Step 2: Definition of maps $\tilde A: X \to
\SL(2,\R)$ and $z: X \to \D$.} Let $\Psi_{q_{i+1}}$ be given by
Lemma~\ref{l.adjust} and put
\begin{multline*}
\big(\tilde A(x), \tilde A(f(x)), \ldots, \tilde A(f^{q_{i+1}-1}(x))\big) = \\
= \Psi_{q_{i+1}}\big(A(x), A(f(x)), \ldots, A(f^{q_{i+1}-1}(x)), z_1(x), z_1(f^{q_{i+1}}(x)) \big),
\end{multline*}
for each $x \in h^{-1}(I)$.
This defines $\tilde A$ on $\bigsqcup_{n=0}^{q_{i+1}-1} f^n(h^{-1}(I))$.
Let $\tilde A$ equal $A$ on the rest of~$X$.

For each $x \in h^{-1}(I)$ and $1 \le n \le q_{i+1} - 1$, let $z(
f^n(x)) = \tilde A^n (x) \cdot z_1(x)$. This defines $z$ on
$\bigsqcup_{n=1}^{q_{i+1}-q_i-1} f^n(h^{-1}(I))$. Let $z$ equal
$z_1$ on the rest of~$X$.

It is easy to see that both maps $\tilde A$ and $z$ are continuous
on the whole $X$, and satisfy $\tilde{A}(x) \cdot z(x) = \tilde{A}
(f(x))$.

\medskip
\noindent \textit{Step 3: Distance estimate.}
To complete the proof, we need to check that $\tilde A$ is $C^0$-close to $A$.
Begin noticing that, by relation~\eqref{e.norm},
$$
B \in \SL(2,\R), \ w \in \D \ \Rightarrow \
d(B \cdot w, 0) \le d(w,0) + 2\log\|B\|.
$$

Now fix $x \in h^{-1}(I)$.
By the definition of $z_0$, we have
$$
d( z_0(x), 0) \le 2(q_{i+i}+q_i) \delta.
$$
If $y = f^{q_{i+1}}(x)$, then $\tau(y)$ equals $1$ or $q_i+1$. In
either case, $\| [A^{\tau(y)}(y)]^{-1}\| = \| A^{\tau(y)}(y)\| \le
e^{(q_i+1)\delta}$. Since $z_1(y) = [A^{\tau(y)}(y)]^{-1} \cdot
z_0(f^{\tau(y)}(y))$, we get
$$
d(z_1(y), 0)
\le d(z_0(f^{\tau(y)}(y)), 0) + 2\log \| A^{\tau(y)}(y) \|
\le 2(q_{i+1} + 2q_i + 1) \delta
$$
Putting things together:
\begin{align*}
d \big( A^{q_{i+1}}(x) \cdot z_0(x), z_1 (f^{q_{i+1}}(x)) \big)
&\le d \big( A^{q_{i+1}}(x) \cdot z_0(x), 0 \big)  + d \big( 0 , z_1 (f^{q_{i+1}}(x)) \big)    \\
&\le d(z_0(x), 0) + 2\log\|A^{q_{i+1}}(x)\| + d (0, z_1 (f^{q_{i+1}}(x)))\\
&\le 2(3q_{i+1}+3q_i+1)\delta.
\end{align*}
By Lemma~\ref{l.adjust},
$$
\|\tilde A A^{-1} - \Id\|_{C^0} \le
\exp \left[\tfrac{1}{2q_{i+1}} d \big( A^{q_{i+1}}(x) \cdot z_0(x), z_1 (f^{q_{i+1}}(x)) \big)\right]-1
< e^{7\delta} - 1.
$$
So $\|\tilde A - A\|_{C^0} \le \delta_0$, as desired.
\end{proof}

\begin{rem}\label{rem.minimal explanation}
A result by Avila and Bochi~\cite{AB minimal} says that
a generic $\SL(2,\R)$-cocycle over a \emph{minimal} homeomorphism
either is uniformly hyperbolic or has uniform subexponential growth \eqref{e.unif subexp growth}
(which is equivalent to the simultaneous vanishing of the Lyapunov exponent for \emph{all} $f$-invariant measures),
provided the space $X$ is compact with finite dimension.
Using this in the place of the aforementioned results from \cite{Bochi} and \cite{Furman},
we obtain the generalization claimed in Remark~\ref{rem.minimal suffices} --
the rest of the proof is the same.
\end{rem}

\subsection{Completion of the Proof of Theorem \ref {t.reduce}} \label{ss.final}

We first prove two lemmas:

\begin{lemma}\label{l.b}
Assume that $A: X \to \SL(2,\R)$ is homotopic to a constant,
and that $(f,A)$ is not uniformly hyperbolic.
Then there exists $\tilde A$ arbitrarily $C^0$-close to $A$ and
$B \in C^0(X,\SL(2,\R))$ such that
$B(f(x))\tilde{A}(x) B(x)^{-1}$ is a constant in $\SO(2,\R)$.
\end{lemma}

\begin{proof}
By Theorem~\ref{t.bounded}, we can perturb $A$ so that there exist
$A_1 \in C^0(X,\SO(2,\R))$ and $B_1 \in C^0(X,\SL(2,\R))$ such
that $A(x) = B_1(f(x)) A_1(x) B_1(x)^{-1}$.

Let $r: \SL(2,\R) \to \SO(2,\R)$ be a deformation retract. Let
$B_2(x) = r (B_1(x))$ and $A_2(x) = B_2(f(x)) A_1(x) B_2(x)^{-1}$.
Then $A_2(x)$ is (i) $\SO(2,\R)$-valued, (ii) conjugate to $A(x)$,
(iii) homotopic to $A(x)$ and therefore to constant.

Due to the existence of the deformation retract $r$, $A_2$ is also
homotopic to a constant as a $X\to\SO(2,\R)$ map. Consider the
covering map $\R \to \SO(2,\R)$ given by $\theta \mapsto
R_\theta$. Let $\phi:X \to \R$ be a lift of $A_2$, that is,
$A_2(x) = R_{\phi(x)}$.

By Lemma~\ref{l.real}, there exists $\tilde \phi$ very close to
$\phi$ such that $\tilde{\phi} = w \circ f - w + a_0$ for some $w
\in C^0(X, \R)$ and $a_0 \in \R$. So the map $\tilde{A}_2 =
R_{\tilde\phi(x)}$ is close to $A_2$ and conjugate to the constant
$R_{a_0}$.

Since $A$ and $A_2$ are conjugate, there exists
$\tilde A$ close to $A$
and conjugate to $\tilde{A}_2$ (and therefore to the constant).
Then $\tilde A$ is the map we were looking for.
\end{proof}

\begin{lemma}\label{l.reverse}
If $A \in \Ruth$, then $(f,A)$ is $\PSL(2,\R)$-conjugate to a
cocycle which is homotopic to a constant.
\end{lemma}

\begin{proof}
Let $\pi : \SL(2,\R) \to \PSL(2,\R)$ be the quotient map.
Since $A\in \Ruth$, there exist $B: X \to \SL(2,\R)$ homotopic to $A$,
$D : X \to \PSL(2,\R)$, and $C \in \PSL(2,\R)$ such that
$\pi(B(x)) = D(f(x)) C D(x)^{-1}$.
The map $x \in X \mapsto D(f(x))^{-1} \pi(A(x)) D(x) \in \PSL(2,\R)$
is homotopic to a constant;
therefore it can be lifted to a map $\tilde A: X \to \SL(2,\R)$,
which is itself homotopic to a constant.
$D$ is a $\PSL(2,\R)$-conjugacy between $A$ and~$\tilde A$.
\end{proof}

\begin{proof}[Proof of Theorem~\ref{t.reduce}]
We have already treated the first part of the theorem, so we will restrict
ourselves to the second part.

Fix $A$ and $A_*$ as in the statement.
By Lemma~\ref{l.reverse}, $(f,A)$ is $\PSL(2,\R)$-conjugate to a cocycle which is homotopic to
a constant.
Since the closure of a $\PSL(2,\R)$ conjugacy class is invariant under
$\PSL(2,\R)$ conjugacies, it is enough to consider the case where $A$ is
homotopic to a constant.  We are going to show that in this case $(f,A)$
lies in the closure of the $\SL(2,\R)$-conjugacy
class of $(f,A_*)$.

By Lemma~\ref{l.b}, $A$ can be perturbed to become conjugate to a
constant $C_* = R_\theta$ in $\SO(2,\R)$. We will explain how to
perturb $C_*$ (and hence $A$, because the conjugacy between $C_*$
and $A$ is fixed) in order that $(f, C_*)$ becomes conjugate to
$(f,A_*)$. There are two cases, depending on $A_*$.

In the case where $A_* = \Id$ or $A_*$ is elliptic (i.e., $|\tr
A_*| < 2$), there exist $B_* \in \SL(2,\R)$ and $\beta \in \R$
such that $A_* = B_*^{-1} R_\beta B_*$. Since $\alpha$ is
irrational, we can choose $k \in \Z$ such that
$2 \pi k\alpha + \theta$
is very close to $\beta$ (modulo $2 \pi \Z$).
We still have the right to perturb
$\theta$, so we can assume $2 \pi k\alpha + \theta = \beta$. Letting
$B(x) = B_* R_{2 \pi kh(x)}$, we see that $B(f(x)) R_\theta B(x)^{-1}$
is precisely $A_*$. So $(f, R_\theta)$ is conjugate to $(f,A_*)$,
as desired.

In the remaining case, $A_*$ is parabolic (i.e., $\tr A_* = \pm
2$) with $A_* \neq \pm \Id$. By the previous case, we can assume
$C_* = \pm \Id$ to start. In fact, we can perturb further and
assume $C_*$ is a parabolic matrix with $C_* \neq \pm \Id$. Then
$C_*$ and $A_*$ are automatically conjugate in the group
$\SL(2,\R)$, so we are done.
\end{proof}

\begin{rem}\label{rem.example}
Assuming that $A$ is homotopic to some cocycle which is conjugate to a constant,
``$\PSL(2,\R)$ conjugacy class'' can be replaced
by ``$\SL(2,\R)$ conjugacy class'' in the second part of Theorem~\ref{t.reduce}.
We give an example showing that the stronger conclusion
does not hold without additional hypotheses.
Let $f: \T^2 \to \T^2$ be given by
$(x,y) \mapsto (x+\alpha, y+2x)$. Let $A(x,y) = R_{2 \pi x}$. We
claim that:
\begin{enumerate}
\item $(f,A)$ is reducible.
\item For any $\tilde A$ close to $A$, $(f,\tilde A)$ is not conjugate to a constant $\SL(2,\R)$-cocycle.
\end{enumerate}
To prove (a), let $D(x,y) = R_{\pi y}$ (which is well-defined in $\PSL(2,\R)$),
and notice $D(f(x,y)) D(x,y)^{-1} = A(x,y)$.
To prove (b), we will show that for any continuous $B: \T^2 \to \SL(2,\R)$,
the map $C(x,y) = B(f(x,y)) A(x,y) B(x,y)^{-1}$ is not homotopic to a constant.
Consider the homology groups $H_1(\T^2) = \Z^2$ and $H_1(\SL(2,\R)) = \Z$, and the
induced homomorphisms
$$
f_*: (m,n) \mapsto (m, 2m+n), \quad
A_*: (m,n) \mapsto m, \quad
B_*: (m,n) \mapsto km + \ell n \, .
$$
We have\footnote{Recall that if $G$ is a path-connected
topological group and $\gamma_1$, $\gamma_2$, $\gamma:[0,1]\to G$
are such that $\gamma(t) = \gamma_1(t) \gamma_2(t)$, then the
$1$-chains $\gamma$ and $\gamma_1+\gamma_2$ are homologous.} $C_*
= B_* \circ f_* + A_* - B_*$, therefore $C_*:(m,n) \mapsto
(2\ell+1)m$ cannot be the zero homomorphism.
\end{rem}

\subsection{The Case of Cantor Groups} \label{ss.cantor groups}

Now assume the second case in Lemma~\ref{l.group dichotomy}.
So there are integers $q_i \to \infty$ and
continuous homomorphisms $h_i: \G \to \T^1$
such that the image of $h_i$ is the (cyclic) subgroup of $\T^1$ of order $q_i$.
Notice that the level sets of $h_i$ are compact, open, and are cyclically permuted by $f$.
They form a tower of height $q_i$ that
will replace the more complicate castle of Figure~\ref{f.tower} in our arguments.
Changing the definition of $h_i$, we can assume that
$h_i (f(x)) = h_i(x) + \frac{1}{q_i} \pmod{1}$.

There are only three proofs that need modification:

\begin{proof}[Proof of Lemma~\ref{l.real} in the Cantor case]
Fix $\phi \in C^0(X,\R)$ with mean zero, and let $\delta>0$ small.
Let $n_0$ be such that $n \ge n_0$ $\Rightarrow$ $|S_n/n|< \delta$ uniformly,
where $S_n$ is the $n$th Birkhoff sum of $f$. Choose $i$ such that
$q_i > n_0$.
Define $\tilde\phi$ and $w : X \to \R$ by
$$
x \in h_i^{-1}(0), \ 0 \le n < q_i \ \Rightarrow \
\tilde \phi (f^n(x)) = \phi(f^n(x)) - \frac{S_{q_i}(x)}{q_i} , \
w(f^n(x)) = S_n(x) - \frac{n S_{q_i}(x)}{q_i}.
$$
Then $\tilde \phi$ and $w$ are continuous, $\tilde \phi = w \circ f - w$,
and $|\tilde \phi - \phi| < \delta$.
\end{proof}

\begin{proof}[Proof of Theorem~\ref{t.bounded} in the Cantor case]
Assume $(f,A)$ is not uniformly hyperbolic.
Given $\delta_0 > 0$, let $\delta>0$ be such that $(e^\delta -1)\|A\|_{C^0} < \delta_0$.
Perturbing $A$, we can assume $\|A^n(x)\| < e^{n\delta}$
for every $n \ge n_0 = n_0(\delta)$.
Fix $q_i > n_0$.
Define $\tilde A: X \to \SL(2,\R)$
so that
$$
\big(\tilde A(x), \tilde A(f(x)), \ldots, \tilde A(f^{q_i-1}(x))\big) =
\Psi_{q_i}\big(A(x), A(f(x)), \ldots, A(f^{q_{i+1}-1}(x)), 0, 0 \big),
\ \forall x \in h_i^{-1}(0),
$$
where $\Psi_{q_i}$ is given by Lemma~\ref{l.adjust}.
Define $z: X \to \D$ by
$z(f^n(x)) = \tilde A ^n (x) \cdot 0$ for $x \in h_i^{-1}(0)$ and
$0 \le n < q_i$.
Then $\tilde A$ and $z$ are continuous, and satisfy
$\tilde A (x) \cdot z(x) = \tilde A(f(x))$.
Moreover, since $d(A^{q_i}(0) , 0) < q_i \delta$,
we have $\|\tilde A - A \| < (e^\delta -1 ) \|A\| < \delta_0$.
\end{proof}

\begin{proof}[Proof of Theorem~\ref{t.reduce}(b) in the Cantor case]
We only need to show that the closure
of the $\SO(2,\R)$-conjugacy class of a constant
$\SO(2,\R)$-valued cocycle contains all constant
$\SO(2,\R)$-valued cocycles.
Given a constant cocycle $R_\theta$, $i \in \N$, and $k \in \Z$,
let $B(x) = R_{2 \pi k h_i(x)}$.
Then $B(f(x)) R_\theta B(x)^{-1} = R_{\theta + 2\pi k/{q_i}}$.
So the claim follows.
\end{proof}

This completes the proofs of Theorems~\ref{t.bounded} and \ref{t.reduce}.

\section{Proof of the Results for Schr\"odinger Cocycles} \label {s.schrodinger}

In this section, we will prove Theorems \ref {t.P}, \ref
{t.regular}, and \ref{t.acspectrum}.

\subsection{Projection Lemma}
The proof of Theorems \ref {t.P} and \ref
{t.regular} is based on the following ``projection
lemma'':

\begin{lemma} \label {generalschrodinger}

Let $0 \leq r \leq \infty$.  Let $f:X \to X$ be a minimal
homeomorphism of a compact metric space with at least three points
(if $r=0$) or a minimal $C^r$ diffeomorphism of a $C^r$ compact
manifold. Let $A \in C^r(X,S)$ be a map whose trace is not
identically zero.  Then there exist a neighborhood $\cW
\subset C^0(X,\SL(2,\R))$ of $A$ and continuous maps
$$
\Phi=\Phi_A:\cW \to C^0(X,S) \quad \text{and} \quad
\Psi=\Psi_A:\cW \to C^0(X,\SL(2,\R))
$$
satisfying:
\begin{gather}
\Phi \text { and } \Psi \text { restrict to continuous maps } \cW
\cap C^s \to C^s, \text { for } 0 \leq s \leq r,  \label{bla1} \\
\Psi(B)(f(x)) \cdot B(x) \cdot \left[\Psi(B)(x)\right]^{-1}=\Phi(B)(x), \label{bla2} \\
\Phi(A)=A \text { and } \Psi(A)=\id. \label{bla3}
\end{gather}
\end{lemma}

\begin{proof}[Proof of Theorems \ref {t.P} and \ref {t.regular}]

The result is easy if $\#X \leq 2$, so we will assume that $\#X
\geq 3$. In this case, Lemma \ref {generalschrodinger} implies the
result unless $\tr A$ is identically $0$.

Assume that $\tr A$ is identically $0$. Let $V \subset X$ be an
open set such that $\overline V \cap f(\overline V)=\emptyset$ and
$\overline V \cap f^2(\overline V)=\emptyset$.  Let $\tilde A \in
C^r(X,S)$ be $C^r$ close to $A$ such that $\tr \tilde A$ is
supported in $V \cup f^2(V)$ and moreover $\tr \tilde A(z)+\tr
\tilde A(f^2(z))=0$ for $z \in V$.  Then $(f,\tilde A)$ is $C^r$
conjugate to $A$: Letting
$B(x)= \id$ for $x \notin f(V) \cup
f^2(V)$, $B(x)=\tilde A(f^{-1}(x)) R_{-\pi/2}$ for $x \in f(V)$, and
$B(x)=R_{-\pi/2} \tilde A(f^{-2}(x))$ for $x \in f^2(V)$, we have
$B(f(x))A(x)B(x)^{-1}=\tilde A(x)$.  If $A$ can be $C^s$
approximated by $\SL(2,\R)$-valued cocycles with property $P$,
then so can $\tilde A$.  Since $\tr \tilde A$ does not vanish
identically, $\tilde A$ can be $C^s$-approximated by $S$-valued
cocycles with property $P$.  Since $\tilde A$ can be chosen
arbitrarily $C^r$ close to $A$, we conclude that $A$ can be
$C^s$-approximated by $S$-valued cocycles with property $P$.
\end{proof}

The proof of Lemma \ref {generalschrodinger} has two distinct
steps.  First we show that $\SL(2,\R)$ perturbations can be
conjugated to localized $\SL(2,\R)$ perturbations and then we show
how to conjugate localized perturbations to Schr\"odinger
perturbations.

In order to be precise, the following definition will be useful.
Given $A \in C^r(X,\SL(2,\R))$ and $K \subset X$ compact, let
$C^r_{A,K}(X,\SL(2,\R)) \subset C^r(X,\SL(2,\R))$ be the set of
all $B$ such that $B(x)=A(x)$ for $x \notin K$.  The two steps we
described correspond to the following two lemmas.

\begin{lemma} \label {localize}

Let $V \subset X$ be any nonempty open set and let $A \in C^r(X,\SL(2,\R))$
be arbitrary.  Then there exist an open neighborhood $\cW_{A,V}
\subset C^0(X,\SL(2,\R))$ of $A$ and continuous maps
$$
\Phi=\Phi_{A,V}:\cW_{A,V} \to C^0_{A,\overline V}(X,\SL(2,\R)) \quad \text{and} \quad
\Psi=\Psi_{A,V}:\cW_{A,V} \to C^0(X,\SL(2,\R))
$$
satisfying
\eqref{bla1}, \eqref{bla2}, and \eqref{bla3}.

\end{lemma}

\begin{lemma} \label {localizeschrodinger}

Let $K \subset X$ be a compact set such that $K \cap
f(K)=\emptyset$ and $K \cap f^2(K)=\emptyset$.  Let $A \in
C^r(X,S)$ be such that for every $z \in K$ we have $\tr A(z) \neq
0$. Then there exists an open neighborhood $\cW_{A,K} \subset
C^0_{A,K}(X,\SL(2,\R))$ of $A$ and continuous maps
$$
\Phi=\Phi_{A,K}:\cW_{A,K} \to C^0(X,S) \quad \text{and} \quad
\Psi=\Psi_{A,K}:\cW_{A,K} \to C^0(X,\SL(2,\R))
$$ satisfying
\eqref{bla1}, \eqref{bla2}, and \eqref{bla3}.

\end{lemma}

Before proving the two lemmas, let us show how they imply Lemma
\ref {generalschrodinger}.

\begin{proof}[Proof of Lemma \ref {generalschrodinger}]

Let $z \in X$ be such that $\tr A(z) \neq 0$.  Let $V$ be an open
neighborhood of $z$ such that with $K=\overline V$ we have $\tr
A(x) \neq 0$ for $x \in K$, $K \cap f(K)=\emptyset$, and $K \cap
f^2(K)=\emptyset$. Let $\Phi_{A,V}:\cW_{A,V} \to
C^0_{A,K}(X,\SL(2,\R))$ and $\Psi_{A,V}:\cW_{A,V} \to
C^0(X,\SL(2,\R))$ be given by Lemma~\ref {localize}.  Let
$\Phi_{A,K}:\cW_{A,K} \to C^0(X,S)$ and $\Psi_{A,K}:\cW_{A,K} \to
C^0(X,\SL(2,\R))$ be given by Lemma~\ref {localizeschrodinger}.
Let $\cW$ be the domain of $\Phi=\Phi_{A,K} \circ \Phi_{A,V}$ and
let $\Psi=(\Psi_{A,K} \circ \Phi_{A,V}) \cdot \Psi_{A,V}$.  The
result follows.
\end{proof}

\begin{proof}[Proof of Lemma \ref {localize}]

For every $x \in X$, let $y=y_x \in V$ and $n = n_x \geq 0$ be such
that $f^n(y)=x$, but $f^{-i}(x) \notin V$ for $0 \leq i \leq n-1$.  Let
$W=W_x \subset V$ be an open neighborhood of $y$ such that $W \cap
f^i(W)=\emptyset$ for $1 \leq i \leq n$.  Let $K=K_x \subset W$
be a compact neighborhood of $y$.
Let $U=U_x \subset f^n(K_x)$ be an open neighborhood of $x$.
Let $\phi=\phi_x:W \to [0,1]$
be a $C^r$ map such that $\phi(z)=0$ for $z \in W \setminus K$,
while $\phi(z)=1$ for $z \in f^{-n}(U)$.

Define maps $\Phi_x,\Psi_x:\cW_x \to C^0(X,\SL(2,\R))$ on some
open neighborhood $\cW_x$ of $A$ as follows.
Let $\Pi:\GL^+(2,\R) \to \SL(2,\R)$ be given by $\Pi(M) = (\det M)^{-1/2} M$.
Let
$\Phi_x(B)(z)=B(z)$ for $z \notin \bigcup_{i=0}^n f^i(W)$,
$\Phi_x(B)(z)=\Pi \left( B(z)+\phi(f^{-j}(z)) (A(z)-B(z)) \right)$
for $z \in f^j(W)$ and
$1 \leq j \leq n$, and
$$
\Phi_x(B)(z) = [\Phi_x(B)(f(z))]^{-1} \cdots [\Phi_x(B)(f^n(z))]^{-1} \cdot
B(f^n(z)) \cdots B(z)
\text{ for $z \in W$.}
$$
Let $\Psi_x(B)(z)=\id$ for
$x \notin \bigcup_{i=1}^{n} f^i(W)$ and
$\Psi_x(B)(z)=\Phi_x(B)(f^{-1}(z)) \cdots \Phi_x(B)(f^{-j}(z))
\cdot B(f^{-j}(z))^{-1} \cdots B(f^{-1}(z))^{-1}$ for $z \in
f^j(W)$ and $1 \leq j \leq n$. Then $\Phi_x$ and $\Psi_x$ are
continuous and have the following properties:
\begin{enumerate}
\item For every $z \in X$, we have $\Psi_x(B)(f(z)) \cdot B(z)
\cdot [\Psi_x(B)(z)]^{-1} = \Phi_x(z)$; \item The set $\{z \in X
\setminus V; \; \Phi_x(B)(z) = A(z)\}$ contains $\{z \in X
\setminus V; \; B(z) = A(z)\} \cup (U_x\setminus V)$; \item
$\Phi_x(A)=A$ and $\Psi_x(A)=\id$.
\end{enumerate}

Choose a finite sequence $x_1,...,x_k$ such that $X =
\bigcup_{i=1}^k U_{x_i}$.  Let $\cW_{A,V} \subset
C^0(X,\SL(2,\R))$ be an open neighborhood of $A$ such that
$\Phi_i=\Phi_{x_i} \circ \cdots \circ \Phi_{x_1}$
is well defined for $1 \leq i \leq k$.  Let
$\Psi_i=\Psi_{x_i} \circ \Phi_{i-1}$ for $2 \leq i \leq k$ and
$\Psi_1=\Psi_{x_1}$. The result follows with $\Phi=\Phi_k$ and
$\Psi=\Psi_k \cdots \Psi_1$.
\end{proof}

\begin{proof}[Proof of Lemma \ref {localizeschrodinger}]

Let $Z \subset S^3$ be the set of all $(B_1,B_2,B_3)$ such that
$\tr B_2 \neq 0$.  One easily checks that $(B_1,B_2,B_3) \mapsto
B_3 B_2 B_1$ is an analytic diffeomorphism between $Z$ and
$$
L = \left\{\begin{pmatrix} a & b \\ c & d \end{pmatrix} \in
\SL(2,\R); \; d \neq 0 \right\}.
$$
Let $\eta:L \to Z$ be the inverse map.

Let $\Phi(B)(x)=A(x)$ if $z \notin \bigcup_{i=-1}^1 f^i(K)$ and
for $z \in K$, let
$$
\left(\Phi(B)(f^{-1}(z)), \Phi(B)(z),\Phi(B)(f(z)) \right)
= \eta \left( B(f^{-1}(z)),B(z),B(f(z)) \right).
$$  Let
$\Psi(B)(z)=\id$ for $z \notin K \cup f(K)$,
$\Psi(z)=\Phi(B)(f^{-1}(z)) \cdot [B(f^{-1}(z))]^{-1}$ for $z \in K$
and $\Psi(z)=\Phi (B)(f^{-1}(z)) \cdot \Phi(B)(f^{-2}(z)) \cdot
[B(f^{-2}(z))]^{-1} \cdot [B(f^{-1}(z))]^{-1}$ for $z \in f(K)$.  All
properties are easy to check.
\end{proof}

\begin{rem}

Let $f:X \to X$ be a homeomorphism of a compact metric space and
let $N \geq n \geq 1$. Let us say that a compact set $K$ is
\emph{$(n,N)$-good} if $K \cap f^k(K)=\emptyset$ for $1 \leq k
\leq n-1$ and $\bigcup_{k=0}^{N-1} f^k(K)=X$.  Then Lemma \ref
{generalschrodinger} holds under the weaker (than minimality of $f$)
hypothesis that there
exist $N \geq 3$ and a $(3,N)$-good compact set $K$ such that $\tr
A(x) \neq 0$ for every $x \in K$.

\end{rem}

\subsection{Dense Absolutely Continuous Spectrum}

To prove Theorem \ref {t.acspectrum}, we
will use the following standard result:

\begin{thm} \label {ac criterion}

Let $f$ be a Diophantine translation of the $d$-torus $\T^d = \R^d / \Z^d$.
Then there exists a set $\Theta \subset \R$ of full Lebesgue
measure such that if $V \in C^\infty(\T^d,\R)$ and $E_0 \in \R$ are
such that $(f,A_{E_0,V})$ is $C^\infty$ $\PSL(2,\R)$ conjugate to
$(f,R_{\pi \theta})$ for some $\theta \in \Theta$, then the
associated Schr\"odinger operator has some absolutely continuous
spectrum.

\end{thm}

For completeness, we will discuss the reduction of this result to
the standard KAM Theorem in more detail in Appendix~\ref{ap.KAM}.

\begin{proof}[Proof of Theorem \ref {t.acspectrum}]

Let $V \in C^0(\T^d,\R)$ be non-constant, and let $E$ be in the
spectrum of the associated Schr\"odinger operator.  By Lemma~\ref{l.b},
there exists $\tilde A \in C^0(\T^d, \SL(2,\R))$ arbitrarily close to $A_{E, V}$
such that $(f, \tilde A)$
is conjugate to $(f,R_{2 \pi \theta})$ for some $\theta \in
\Theta$. Approximating the conjugacy by a $C^\infty$
map, we may assume that $\tilde A$ is $C^\infty$.
Applying Lemma \ref {generalschrodinger} with $r=\infty$,
we find a $C^\infty$ function $\tilde V$ which is $C^0$-close to $V$,
such that $(f,A_{E,\tilde V})$ is conjugate to $(f,A)$, and hence to
$(f,R_{2 \pi \theta})$. The result now follows from Theorem \ref {ac
criterion}.
\end{proof}

\appendix

\section{Topological Groups}\label{ap.groups}

We quickly review some material that can be found in~\cite{Rudin}.
Let $G$ be a topological group. If $G$ is locally compact and
abelian, one defines the \emph{dual group} $\Gamma = \hat{G}$; it
consists of all characters of $G$ (i.e., continuous homomorphisms
$\gamma: G \to \T^1$). Then $\Gamma$ is an abelian group, and with
the suitable topology, it is also locally compact. Some important
facts are: (1)~$G$ is compact iff $\Gamma$ is discrete;
(2)~Pontryagin Duality\footnote{``$=$'' means isomorphic and
homeomorphic.}: $\hat{\Gamma} = G$.

\begin{proof}[Proof of Lemma~\ref{l.group dichotomy}]
Since $\G$ is compact and infinite, the dual group $\Gamma$ is discrete and infinite.

First, assume that $\Gamma$ contains an element of infinite order.
So we can assume $\Z$ is a closed subgroup of $\Gamma$. Let
$\iota: \Z \to \Gamma$ be the inclusion homomorphism, and let $s :
\hat{\Gamma} \to \hat{\Z}$ be its dual\footnote{The dual of a
continuous homomorphism $h:G_1 \to G_2$ is the continuous
homomorphism ${\hat{h}: \hat{G_2} \to \hat{G_1}}$ defined by
$\langle x_1, \hat{h}(\gamma_2)\rangle =\langle h(x_1),
\gamma_2\rangle$, where $x_1 \in G_1$, $\gamma_2 \in \hat{G_2}$.}.
Then $s$ is onto: every character on $\Z$ can be extended to a
character on $\Gamma$; see \cite[\S2.1.4]{Rudin}. Since
$\hat{\Gamma} = \G$ and $\hat{\Z} = \T^1$, alternative~(a) holds.

Now, assume that all elements of $\Gamma$ are of finite order.
Then $\G$ is a Cantor set, see \cite[\S2.5.6]{Rudin}.
There is a translation $x \mapsto x+ \alpha$ of $\G$
which is a factor of the minimal homeomorphism $f$, and so it is itself minimal.
Therefore $\Gamma$ is a subgroup of $\T^1_d$,
the circle group with the discrete topology; see \cite[\S2.3.3]{Rudin}.
So there exist cyclic subgroups
$\Lambda_i \subset \Gamma$ with
$|\Lambda_i| \to \infty$.
Let $H_i\subset \G$ be the annihilator (see \cite[\S2.1]{Rudin})
of $\Lambda_i$.
Then $\G/H_i = \Lambda_i$,
and the quotient homomorphism $\G \to \G/H_i$ is continuous.
So we are in case~(b).
\end{proof}

\section{Absolutely Continuous Spectrum via KAM}\label{ap.KAM}

In this section, $f:\T^d \to \T^d$ is a minimal translation of the torus
$\T^d=\R^d/\Z^d$.  We will show how Theorem \ref {ac criterion}
reduces to a result of \cite {H}, based on the usual KAM theorem.
We will use the following criterion for absolutely continuous
spectrum (see, e.g., \cite{Simon} for a simple
proof):

\begin{lemma} \label {sigmab}

Let $f:\T^d \to \T^d$ be a homeomorphism, and let $V \in
C^0(\T^d,\R)$. If $\Sigma_b=\{E \in \R,\, (f,A_{E,V}) \text { is
conjugate to a cocycle of rotations}\}$ has positive Lebesgue
measure, then the associated Schr\"odinger operators have some
absolutely continuous spectrum.

\end{lemma}

Recall that if $A:\T^d \to \SL(2,\R)$ is homotopic to a constant, one
can define a fibered rotation number $\rho(f,A) \in \R/\Z$ (see
\cite {H}). The following properties of the fibered rotation
number are easy to check:
\begin{enumerate}

\item If $A$ is a constant rotation of angle $\pi \theta$, then
$\rho(f,A)=\theta$.

\item If $B$ is a $\PSL(2,\R)$ conjugacy between $(f,A)$ and
$(f,A')$, then $\rho(f,A)=\rho(f,A')+k \alpha$, where $k=k(B) \in
\Z$ only depends on the homotopy class of $B$.

\item The fibered rotation number is a continuous function of $A
\in C^0(\T^d,\SL(2,\R))$.

\item If $A \in C^0(\T^d,\SL(2,\R))$ is $\PSL(2,\R)$ conjugate to a
constant rotation, then the fibered rotation number (as a function
of $C^0(\T^d,\SL(2,\R))$ is $K$-Lipschitz at $A$ for some
$K>0$.\footnote { Let $\phi:Y \to Z$ be a function between metric
spaces and let $K>0$.  We say that $\phi$ is $K$-Lipschitz at $y
\in Y$ if there exists a neighborhood $\cV \subset Y$ of $y$ such
that for every $z \in \cV$, $d_Z(\phi(z),\phi(y)) \leq K
d_Y(z,y)$.}

\end{enumerate}

In \cite {H}, it is shown how the KAM Theorem implies reducibility
for cocycles close to constant, under a Diophantine assumption on
$f$ and on the fibered rotation number.  To state it precisely, it
will be convenient to introduce the following notation.

Let $n \geq 1$ and $\kappa,\tau>0$.  Let $\DC_{n,\kappa,\tau}$ be
the set of all $\alpha \in \R^n$ such that there exists
$\kappa,\tau>0$ such that for every $k_0 \in \Z$, $k \in \Z^n
\setminus \{0\}$, we have
$\left|k_0+\sum_{i=1}^n k_i \alpha_i\right|>\kappa
\left(\sum_{i=1}^n |k_i|\right)^{-\tau}$.  Notice that $\DC_{n,\kappa,\tau}$
is $\Z^d$ invariant.  Let $\DC_n=\bigcup_{\kappa,\tau>0}
\DC_{n,\kappa,\tau}$.  We say that $f$ is a Diophantine
translation if $f(x)=x+\alpha_f$ for some $\alpha_f \in \DC_d$.

For every $\alpha \in \R^d$, let $\Theta_{\alpha,\kappa,\tau}$ be
the set of all $\theta \in \R$ such that $(\alpha,\theta) \in
\DC_{d+1,\kappa,\tau}$. Notice that $\Theta_{\alpha,\kappa,\tau}$
is $\Z$ invariant. Let $\Theta_\alpha=\bigcup_{\kappa,\tau>0}
\Theta_{\alpha,\kappa,\tau}$. Then $\Theta_\alpha=\emptyset$ if
$\alpha \notin \DC_d$ and $\Theta_\alpha$ has full Lebesgue
measure if $\alpha \in \DC_d$.  Moreover, every $\theta \in
\Theta_\alpha$ is a Lebesgue density point of
$\Theta_{\alpha,\kappa,\tau}$ for some $\kappa,\tau>0$.

\begin{thm}[\cite {H}, Corollaire 3, page 493 and Remarque 1, page 495]
\label {conjugation}

For every $\theta \in \R$ and $\kappa,\tau>0$, there exists a
neighborhood $\cW \subset C^\infty(\T^d,\SL(2,\R))$ of $R_\theta$
such that if $A \in \cW$, and $\rho(f,A) \in
\Theta_{\alpha_f,\kappa,\tau}/\Z$, then $(f,A)$ is $C^\infty$
conjugate to a constant rotation.

\end{thm}

\begin{proof}[Proof of Theorem \ref {ac criterion}]

Let $\Theta=\Theta_{\alpha_f}$.  Since $\alpha_f \in \DC_d$,
$\Theta$ has full Lebesgue measure. Let $V$, $E_0$ and $\theta$ be
as in the statement of the theorem. Let $\kappa,\tau>0$ be such
that $\theta$ is a Lebesgue density point of
$\Theta_{\alpha_f,\kappa,\tau}$.

Let $\Sigma_r$ be the set of all $E \in \R$ such that
$(f,A_{E,V})$ is $C^\infty$ conjugate to a constant rotation.

Let $k \in \Z$ be such that $\rho(f,A_{E_0,V})=\theta+k\alpha$. By
Theorem \ref {conjugation}, there exists an open interval $I$
containing $E_0$ such that if $E' \in I$ and $\rho(f,A_{E',V})-k
\alpha \in \Theta_{\alpha_f,\kappa,\tau}$ then $E' \in \Sigma_r$.
Let $\rho:I \to \R/\Z$ be given by $\rho(E)=\rho(f,A_{E,V})-k
\alpha$.  If $\rho(I)=\{\theta\}$ (this cannot really happen, but
we do not need this fact), then $I \subset \Sigma_r$. Otherwise,
by continuity of the fibered rotation number, $\rho(I \cap
\Sigma_r) \supset \rho(I) \cap \Theta_{\alpha,\kappa,\tau}/\Z$ has
positive Lebesgue measure.  Since $\rho$ is $K(E)$-Lipschitz at
every $E \in \Sigma_r$, we conclude in any case that $\Sigma_r$
has positive Lebesgue measure.  The result follows by Lemma \ref
{sigmab}.
\end{proof}

\begin{ack}
We thank the referee for several suggestions that improved the exposition.
\end{ack}


\end{document}